\documentclass{amsart}
\usepackage[utf8]{inputenc}

\usepackage{amsmath,amssymb,amsthm}
\usepackage{comment}
\usepackage{hyperref}
\usepackage{pgf,tikz}
\usepackage{esint}
\usepackage{bbm}

\usetikzlibrary{arrows}
\usetikzlibrary{arrows.meta}

\newcommand{\RR}{\mathbb R}

\newcommand{\NN}{\mathbb N}

\renewcommand{\H}{\mathcal H}
\renewcommand{\L}{\mathcal L}
\newcommand{\eps}{\varepsilon}
\newcommand{\m}{\mathfrak m}
\renewcommand{\d}{\mathrm{d}}

\newcommand{\defeq}{:=}

\newcommand{\enclose}[1]{\left(#1\right)}

\renewcommand{\H}{\mathcal{H}}

\newcommand{\diam}{\mathrm{diam}}

\newtheorem{theorem}{Theorem}[section]
\newtheorem{proposition}[theorem]{Proposition}

\newtheorem{lemma}[theorem]{Lemma}
\newtheorem{corollary}[theorem]{Corollary}

\theoremstyle{definition}
\newtheorem{definition}[theorem]{Definition}
\theoremstyle{remark}
\newtheorem{remark}[theorem]{Remark}

\author[Novaga]{Matteo Novaga} 
\address[Matteo Novaga, Emanuele Paolini, Eugene Stepanov]{Dipartimento di Matematica, Universit\`a di Pisa, 
	Largo Bruno Pontecorvo 5 \\ I-56127, Pisa}
\address[Eugene Stepanov] {%
 St.Petersburg Branch of the Steklov Mathematical Institute of the Russian Academy of Sciences,
  St.Petersburg, Russia
 \and
  Department of Mathematical Physics, Faculty of Mathematics and Mechanics,
  St. Petersburg State University, Universitetskij pr.~28, Old Peterhof,
  198504 St.Petersburg, Russia
  \and ITMO University
  \and Faculty of Mathematics, Higher School of Economics, Moscow
}

\email[Matteo Novaga]{matteo.novaga@unipi.it}

\author[Paolini]{Emanuele Paolini} 
\email[Emanuele Paolini]{emanuele.paolini@unipi.it}

 \author[Stepanov]{Eugene Stepanov}
 \email[Eugene Stepanov]{stepanov.eugene@gmail.com}

\thanks{
 This research was partially supported by MUR Excellence Department Project 
 awarded to the Department of Mathematics of the University of Pisa.
 The work of M.N.\ was partially supported by Next Generation EU, Mission 4, Component 2, PRIN 2022E9CF89.
 The work of E.P.\ was partially supported by Next Generation EU, Mission 4, Component 2, PRIN 2022PJ9EFL,
 and by project PRA 2022 14 GeoDom (Università di Pisa).
 The work of E.S.\ is partially within the framework of HSE University Basic Research Program.
The authors are indebted to Francesco Nobili for suggesting a reference 
 that improved substantially the result of Lemma~\ref{lemma:PI}%
 }

\date{\today}

\title{On the total surface area of potato packings}

\begin{document}

\begin{abstract}
We prove that if we fill without gaps a bag
with infinitely many potatoes, in such a way that they touch each other in
few points, 
then the total surface area of the potatoes must be infinite.
In this context potatoes are measurable subsets of the Euclidean 
space, the bag is any open set of the same space.
As we show, this result also holds in the general context
of doubling (even locally) metric measure spaces satisfying Poincar\'e inequality,
in particular in smooth Riemannian manifolds and even in some sub-Riemannian spaces.
\end{abstract}

\maketitle

\tableofcontents

\section*{Introduction}

\begin{figure}
\begin{center}
\includegraphics[height=0.5\textwidth]{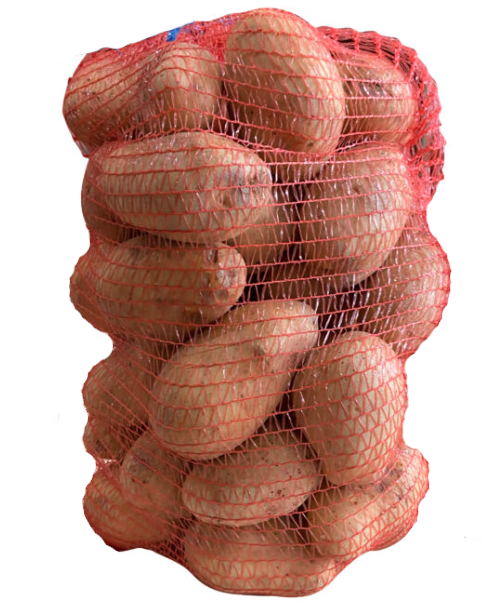}
\hfill
\includegraphics[height=0.5\textwidth]{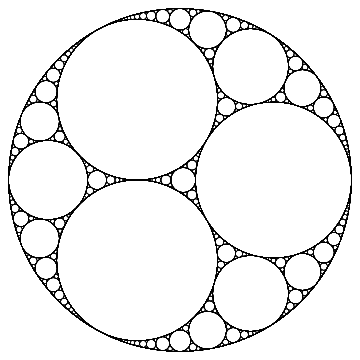}
\end{center}
\caption{On the left-hand side: a potato bag.
On the right-hand side: an Apollonian gasket.
}
\label{fig:bag}
\end{figure}

In this note we prove the following curious fact: 
if we fill without gaps a bag 
(modeled, say, by an arbitrary open connected subset 
of $\RR^d$) with potatoes 
(modeled, in this case, by open sets of finite perimeter),
in such a way that the potatoes touch each other in 
a single point (more generally in a set of zero surface measure), 
see Figure~\ref{fig:bag},
then the total surface area of the potatoes is infinite.
We show that this very simple result is in fact true for a 
fairly large class of metric measure spaces, where the perimeter 
of a set can be naturally defined.

This result implies in particular that the residual set 
(the complement of the union of the potatoes)
has Hausdorff dimension at least $d-1$.
Theorem~1.4(ii) from \cite{MaiNta22} shows that
this statement is sharp: in fact it proves the existence
of a packing of a planar convex set by planar strictly convex 
sets for which the dimension of the residual set is 
exactly $1$.

In this way we generalise the results of \cite{MaiNta22},
dedicated to packing of convex sets, 
which in their turn generalise those from the series of papers 
\cite{Lar66,Lar66b,Lar66c,Lar68}
and \cite{MikMik02}
dedicated to packings of spheres.

\section{Notation and preliminary results}

We will consider here metric measure spaces $(X,\d,\m)$
where $X$ is a nonempty set equipped with a distance $\d$ 
and a $\sigma$-finite Borel measure $\m$ with $\m(X)\neq 0$.
We say that a function $F$ defined on the borel sets of a metric space $X$
with values in $[0,+\infty]$
is a \emph{perimeter-like evaluation} if it satisfies the following properties:
\begin{enumerate}
  \item[(0)] $F(\emptyset)=0$;
%  \item[(S)] $F(A\cup B)\le F(A)+F(B)$;
  \item[(T)] $F(A) \ge \limsup_n F(A_n)$ whenever $A_n\subset A$ and $F(A\setminus A_n)\to 0$; 
  \item[(C)] $F(X\setminus A) = F(A)$;
  \item[(L)] $F(A)\le \liminf_n F(A_n)$ if $\lim_n \m(A_n\triangle A)=0$
  (one says that $F$ is lower semicontinuous with respect to $L^1(X,\m)$ convergence of 
  sets);
  \item[(Z)] if $\m(A \triangle B)=0$ then $F(A)=F(B)$.
\end{enumerate}

\begin{lemma}\label{lemma:T}
The property (T) is valid if, for some $c>0$, one has
\begin{enumerate}
  \item[(T')] 
  $F(A) \ge F(A\setminus B) - c F(B)$ whenever $B\subset A$.
\end{enumerate}
\end{lemma}
\begin{proof}
Just notice that 
\begin{align*}
  F(A) &= F(A_n \cup (A\setminus A_n))\\
  &\ge F(A_n) - cF(A\setminus A_n) && \text{by $(T')$}
\end{align*}  
and if $F(A\setminus A_n)\to 0$, then we have (T).
\end{proof}

%\section{PI-spaces and perimeter}
For a fairly general class of metric measure spaces, 
namely those with a doubling measure and satisfying 
the Poincar\'e inequality, 
a perimeter functional has been defined 
in \cite{Mir03,AmbMirPal04,BonPasRaj20}.
% by 
% \[
% P(E) \defeq \inf \liminf_{n\to\infty}\int{\rm lip}(f_n)\,\d{\m},
% \]
% where the infimum is taken among all sequences $f_n$ 
% of locally lipschitz functions such that $f_n \to \1_E$ in $L^1(X,\m)$ as $n\to \infty$. 
% Here, by locally Lipschitz, we mean that for every point there is a neighborhood 
% where the restriction is Lipschitz and ${\rm lip}(f)$ 
% stands for \emph{local lipschitz constant} of a function $f \colon X\to \RR$. 
% By localizing the above relaxation procedure to open sets and by a Carath\'eodory 
% construction, see \cite{Mir03,AmbDiM14}, 
% it can be shown that there is an underlying \emph{perimeter measure} 
% $B\mapsto P(E,B)$ on Borel sets $B\subset X$ so that $P(E) = P(E,X)$.
%
% We shall consider the class of spaces satisfying a doubling and Poincar\'e condition, the so-called PI-spaces. 
We refer to the monograph \cite{HeiKosShaTys15} and references therein for a detailed discussion
of such spaces.
Here we just recall the basic definitions. 
A measure ${\mathfrak m}$ is said \emph{doubling}, 
provided there is a $C_D\ge 1$ such that
\[
   {\mathfrak m}(B_{2r}(x))\le C_D \m(B_r(x)),\qquad \text{for all }x\in X,r \in (0,R). 
\]
We will shortly say that $(X,\d,{\mathfrak m})$ is a doubling space.

We say that a metric measure space $(X,\d,{\mathfrak m})$ supports a 
\emph{weak $(1,1)$ Poincar\'e} inequality,  
provided there is $\tau_P\ge 1$, and a constant $C_P>0$ such that, 
for any locally Lipschitz function $f\colon X\to \RR$, one has
\[
\fint_{B_r(x)}|f-f_{x,r}|\,d\m \le C_P r\fint_{B_{\tau_P r}(x)} {\rm Lip}(f)\,d\m,
  \qquad \text{for all }x \in X, r>0,
\]
where $f_{x,r}\defeq \fint_{B_r(x)} f\,d{\mathfrak m}$ 
and $Lip(f)$ stands for the Lipschitz constant of $f$ over $B_{\tau_P r}(x)$.
Actually, this is usually formulated 
with \emph{upper gradients} (see \cite{HeiKos98}), 
rather than with local Lipschitz constant. 
However, if $(X,\d,\m)$ is doubling, 
the two formulations are equivalent \cite{Kei03}.
\begin{definition}
A PI-space is a metric measure space $(X,\d,{\mathfrak m})$ that is doubling and supports 
a weak $(1,1)$-local Poincar\'e inequality. 
\end{definition}
%This includes, of course, the classical case $X=\RR^d$ with 
%the Euclidean distance $\d$ and the Lebesgue measure $\m$.
The perimeter functional in a PI-space has been defined in \cite{Mir03,AmbMirPal04} (see also \cite{BonPasRaj20}) 
as a natural generalization of the classical Euclidean perimeter in $\RR^d$ (see \cite{Giu84}).
Note that, like in the classical Euclidean case $\RR^d$ where the perimeter 
is the $(d-1)$-dimensional Hausdorff measure of the 
essential boundary, also 
in PI-spaces the perimeter of a set $E$ 
is a measure of the essential boundary $\partial^e E$ of $E$,
defined in \cite{AmbMirPal04}, and is absolutely continuous with respect to the 
so-called codimension one Hausdorff measure $\H^{-1}$, 
with a Borel density $\theta_E\colon X\to [\alpha,\beta]$
where $\alpha$ and $\beta$ are two positive constants depending only on 
the constants $C_D$, $C_P$, $\tau_P$ of the space
%but the density $\theta_E$ itself may depend on the set $E$
\cite{BonPasRaj20}.
In particular, one has the integral representation
\[
  P(E) = \int_{\partial^e E} \theta_E(x)\, d\H^{-1}(x).
\]
% If, moreover,
% \[
%   \theta_E = \theta_F\qquad \mathcal{H}^{-1}\text{-a.e.\ on } 
%   \partial^eE\cap \partial^e F, 
% \]
% whenever $E,F\subset X$ are two sets of finite perimeter satisfying $E\subset F$,
% then the PI-space $(X,\d,\m)$ is called \emph{isotropic}.
We refer to \cite{BonPasRaj20,BreNobPas23} 
where this notion has been further studied.
It is important to know that this includes many classical examples,
like for instance
\begin{enumerate}
  \item the Euclidean space $X=\RR^d$ (or $X$ open subset of $\RR^d$) with the Lebesgue measure 
  (in this case $\H^{-1}$ is the classical $(d-1)$-dimensional Hausdorff measure and $\partial^e E$ 
  is equivalent to the reduced boundary of $E$ as defined by De Giorgi),
  \item $X$ a finite dimensional space equipped with any anisotropic norm and the Lebesgue measure,
  %or a smooth Finsler manifold with the volume measure, 
  or even 
  some of the so-called $RCD$ metric measure spaces,
  \item $X$ a Heisenberg group of any dimension, or even some of the more general Carnot groups.
\end{enumerate}

\begin{lemma}\label{lemma:PI}
Let $(X,\d,\m)$ be a $PI$ space.
Then the perimeter functional $P$ 
is a perimeter-like evaluation, i.e.\ it satisfies 
all the properties (0), (T), (C), (L), and (Z).
\end{lemma}
\begin{proof}
Property $(0)$ follows from $\partial^e \emptyset = \emptyset$.
Property $(T')$, with $c=1$, is shown in \cite[eq.~(4.1)]{Amb02}.
Proposition~1.7 (i), (ii), and (vi) in \cite{BonPasRaj20} shows
(Z), (L), and (C) respectively.
\end{proof}

%For the definition and properties of
%the perimeter, $\H^{-1}$ and the essential boundary $\partial^e E$ 
%in metric measure spaces
%we refer to \cite{AmbMirPal04,BonPasRaj20}.

%We denote by $\RR^d$ the $d$-dimensional euclidean space,
%by $\H^{d-1}$ the $(d-1)$-dimensional Hausdorff measure,
%by $1_E$ the characteristic function of a measurable set $E\subset \RR^d$,
%%by $D1_E$ its distributional derivative 
%and by $P(E,A)$ 
%the Caccioppoli perimeter of $E$ relative to an open set $A\subset \RR^d$.
%We recall that 
%when $P(E,A)$ is finite then $D1_E$ is a vector valued Radon measure
%over $A$ and $P(E,A)=\abs{D1_E}(A)$ where $\abs{D1_E}$ stands 
%for the total variation measure of the vector valued measure $D1_E$.
%
%*** definire $\m$ e supporre $\m(X)\neq 0$. ***

\section{Main result}

We start with the following auxiliary statement.

\begin{proposition}
If $F$ satisfies properties (0), (C), and the family of sets 
$E_k$, $k\in \NN$ is such that 
$\bigcup E_i = X$,
and 
\[
   F\enclose{\bigcup_j E_j} \ge \sum_j F(E_j)
\]
then $F(E_i) = 0$ for all $i\in\NN$.

If, additionally, $F$ satisfies (Z) then the same holds true if 
$\m(X\setminus\bigcup E_i) = 0$
rather than $\bigcup E_i = X$.
\end{proposition}
\begin{proof}
  One has
\begin{align*}
  0 
  &= F(\emptyset) && \text{by (0)}\\
  &= F(X)  && \text{by (C)}\\
  &= F\enclose{\bigcup E_i} && \text{by assumption}\\
  &\ge \sum_i F(E_i) && \text{by assumption}
\end{align*}
hence $F(E_i)=0$ for all $i\in \NN$.

If, additionally, $F$ satisfies (Z),
and $\m(X\setminus\bigcup E_i) = 0$
one has
\begin{align*}
  0 
  &= F(\emptyset)=F(X)\\
  &= F\enclose{X \setminus \bigcup E_i} && \text{by (Z)}\\
  &= F\enclose{\bigcup E_j} && \text{by (C)}\\
  &\ge \sum F(E_j) && \text{by assumption}
\end{align*}
so that $F(E_j)=0$ for all $j$.
\end{proof}

The Theorem below is the main tecnical result which will be further 
translated into corollaries adapted to applications.

\begin{theorem}\label{th:main}
If $F$ satisfies properties (0), (C), (T), (L) and the family of sets 
$E_k$, $k\in \NN$ is such that 
%$E_i\cap E_j=\emptyset$ and
$F(E_i\cup E_j) = F(E_i) + F(E_j)$
for $i\neq j$,
$\bigcup_{i=0}^{+\infty} E_i = X$,
and $\m\enclose{\bigcup_{i=1}^{+\infty} E_i} < +\infty$,
then 
either $\sum_{i=0}^{+\infty} F(E_i) = +\infty$ or
$F(E_i) = 0$ for all $i\in\NN$.

If, additionally, $F$ satisfies (Z) then the same holds true if 
$\m(X\setminus\bigcup_{i=0}^{+\infty} E_i) = 0$
%and $\m(E_i\cap E_j)=0$ 
for all $i\neq j$,
in place of $\bigcup E_i = X$.
% and $E_i\cap E_j=\emptyset$.
\end{theorem}
\begin{proof}
Let 
\[
  T_n\defeq \bigcup_{i=n+1}^{+\infty} E_j, \qquad
  T_n^m\defeq \bigcup_{i=n+1}^{m} E_j.
\]
Clearly $T_n^m \nearrow T_n$ as $m\to +\infty$, 
hence 
$\lim_m \m(T_n^m) = \m(T_n)$. 
Since $\m(T_n)\le \m\enclose{\bigcup_{i=1}^{+\infty} E_i}<+\infty$
then we have that $\m(T_n^m \triangle T_n)\to 0$ 
%in $L^1(X,\m)$
as $m\to+\infty$.
Therefore
\begin{equation}\label{eq:48972345}
\begin{aligned}
  F(T_n) &\le \liminf_m F(T_n^m) && \text{by (L)}\\
    &= \liminf_m \sum_{i=n+1}^m F(E_i) &&\text{by assumption}\\
    &= \sum_{i=n+1}^{+\infty} F(E_i).
\end{aligned}
\end{equation}
If $\sum_{i=1}^{+\infty} F(E_i)=+\infty$ the proof is concluded.
Otherwise the righthand side of \eqref{eq:48972345} is the remainder 
of a convergent number series and hence $F(T_n)\to 0$ as $n\to +\infty$.
We have then
\begin{align*}
  F\enclose{\bigcup_{i=0}^{+\infty} E_i}
  &\ge \limsup_n F\enclose{\bigcup_{i=0}^n E_i} && \text{by (T)}\\ 
  &= \lim_n \sum_{i=0}^n F(E_i) && \text{by assumption}\\
  &= \sum_{i=0}^{+\infty} F(E_i).
\end{align*}
The proof is then concluded using the previous proposition.
\end{proof}

\section{Some applications}

Combining Theorem~\ref{th:main} with Lemma~\ref{lemma:PI} we obtain the following result.

\begin{theorem}\label{th:1}
  Let $(X,\d,\m)$ be a $PI$ space.
  Let $E_k$, $k\in \NN$, be a sequence
  of Borel subsets of $X$ 
  such that 
  \begin{enumerate}
      \item[(i)] $\m(E_k \cap E_j)=0$ for $k\neq j$;
      \item[(ii)] $\m\enclose{X \setminus \bigcup E_k}=0$;
      \item[(iii)] $\H^{-1}(\partial^e E_k \cap \partial^e E_j )=0$ for all $k\neq j$.
      \item[(iv)] $\m(E_0)>0$ and $\m (E_1)>0$.
  %    \item $P(E_k\cup E_j)\ge P(E_k) + P(E_j)$ for all $k\neq j$.
  \end{enumerate}
  Then 
  \[
      \sum_k P(E_k) = +\infty.
  \]
\end{theorem}
\begin{proof}
  Conditions (i) and (iii) imply
  \[
    P(E_k\cup E_j) = P(E_k) + P(E_j)
  \]
  in view of Lemma~2.3(ii) in \cite{BonPasRaj20}.
  The assumptions of Theorem~\ref{th:main} are therefore satisfied by $P$ 
  in view of Lemma~\ref{lemma:PI}.
  Hence by Theorem~\ref{th:main} we conclude that either $\sum P(E_k)=+\infty$
  or $P(E_k)=0$ for all $k$.
  But the latter option is excluded by (iv) in view of the relative isoperimetric 
  inequality for $PI$ spaces, provided in \cite{Mir03}.
  %theorem~1.17 of \cite{BonPasRaj20}.
\end{proof}

\begin{corollary}\label{co:2}
Let $\m$ be the Lebesgue measure on $\RR^d$.
Let $\mathcal F$ be a family of at least two 
%(but possibly infinitely many) 
measurable nonempty subsets of $\RR^d$, 
each with positive volume $\m$ and kissing each other 
on sets of zero $\H^{d-1}$ measure, i.e. 
\begin{equation}\label{eq:0967}
  \H^{d-1}(\bar A\cap \bar B)=0, \qquad \text{for every}\ A,B\in \mathcal F, A\neq B.
\end{equation}
In particular this assumption is satisfied 
if the sets in $\mathcal F$ are strictly convex and their interiors 
are not intersecting.

Let $B\subset \RR^d$ be an open ball.
If 
$\sum_{E\in \mathcal F} \H^{d-1}(\partial E\cap B)<+\infty$,
then either there exists a set $E\in \mathcal F$ such that
\[
  \m (B\setminus E)=0
\]
or
\[
  \m \enclose{B \setminus \bigcup_{E\in \mathcal F} E}>0,
\]
i.e.\ there is a subset of $B$ with positive 
measure which is not covered by the union of $\mathcal F$.

In particular if $\Omega\subset \RR^d$ is an open set 
%(a ``bag of potatoes''),
and all $E \in \mathcal F$ are open nonempty subsets of $\Omega$,
regular in the sense that they coincide with the interior of their closure,
satisfying~\eqref{eq:0967} and
\begin{equation}\label{eq:m}
\m\left(\Omega\setminus \bigcup_{E\in \mathcal F} E\right)=0,
\end{equation}
%for some open $\Omega\subset \RR^n$
%and 
then given any $x\in \Omega \cap \bigcup_{E\in \mathcal F} \partial E$ 
and $B$ an open ball centered at $x$, one has
\begin{equation}\label{eq:68}
  \sum_{E\in \mathcal F} \H^{d-1}(\partial E\cap B)=+\infty.
\end{equation}
If, moreover, $\Omega$ is connected, this implies in particular
\begin{equation}\label{eq:685}
  \sum_{E\in \mathcal F} \H^{d-1}(\partial E\cap \Omega)=+\infty.
\end{equation}
%If $d=2$ and all $E\in\mathcal F$ are convex, 
%this also implies
%begin{equation}\label{eq:69}
% \sum_{E\in \mathcal F} \diam E = +\infty.
%end{equation}
\end{corollary}
\begin{proof}
We apply Theorem~\ref{th:1} with $X\defeq B$, equipped with the Euclidean distance,
and $P$ is the usual Euclidean (Caccioppoli) perimeter relative to $B$.
Clearly, the doubling condition holds for $B$ and,
since $B$ is connected, 
the Poincar\'e inequality is satisfied and $X=B$ is a $PI$ space.
Then, the De Giorgi reduced boundary $\partial^*E$ of a Borel set $E\subset \RR^d$ 
satisfies $\partial^* E \subset \partial^e E$ and 
$\H^{d-1}(\partial^e E) = \H^{d-1}(\partial^* E)$.
%where $\H^{n-1}$ stands for the $(n-1)$-dimensional Hausdorff measure.
Notice that $\H^{d-1}$ coincides with the spherical Hausdorff 
measure $\mathcal S^{d-1}$ on rectifiable sets, and
(see the Example ``weighted spaces'' in Section~7 of \cite{AmbMirPal04})
$\H^{d-1}$ coincides, up to a multiplicative constant, with $\mathcal S^{d-1}$.

If there exists $E\in \mathcal F$ such that $\m(B\setminus E)=0$ or, 
if $\m(B\setminus \bigcup_{E \in \mathcal F} E)>0$ there is nothing to prove.
Otherwise there should be at least two different sets $E_0$, $E_1$ in $\mathcal F$  
such that $\m(B\cap E_0)>0$ and $\m(B\cap E_1)>0$.
Enumerate now all the sets 
of $\mathcal F$ as $E_k$, $k\in \NN$
(clearly $\mathcal F$ is at most countable since each set is assumed to have positive measure,
and in the case that $\mathcal F$ is finite we can complete the sequence 
with empty sets).
Notice that $\H^{d-1}(\partial^e E_k\cap \partial^e E_j)=0$ for all $k\neq j$
and $\m(E_k\cap E_j)=0$ since $\H^{d-1}(\bar E_k \cap \bar E_j)=0$.

Therefore, we can apply Theorem~\ref{th:1}
to the sequence $E_k\cap B$ to get
$\sum_k P(E_k, B)=+\infty$
hence 
\[
\sum_k \H^{d-1}(\partial E_k\cap B)
\ge 
\sum_k \H^{d-1}(\partial^* E_k\cap B)
= \sum_k P(E_k,B) = +\infty,
\]
showing the first claim.

In the particular case when $\m(\Omega\setminus \bigcup_{E\in \mathcal F} E)=0$
for some open $\Omega\subset \RR^d$, taking $x$ and $B$ as in the statement,
we consider a ball $B'\subset B$ centered at $x$ such that 
$B'\subset \Omega$.
%Fix $E$ in $\mathcal F$ such that $x\in \partial E$. 
%Then there exists $F\neq E$ in $\mathcal F$ such that $\m(B'\cap F)>0$,
%because if otherwise $\m(B'\cap F)=0$ then $B'\cap F=\emptyset$ 
%(being $F$ open) and hence $E\cap B'=B'$ and $x\not \in \partial E$.
Then, for every $E\in \mathcal F$ either $E\cap B'=\emptyset$ or
$E\cap B'\neq \emptyset$. 
In the latter case
one has $\partial E\cap B'\neq \emptyset$ since 
otherwise we would have $B'\subset E$ which contradicts the choice of the center $x$
and the fact that all the sets in $\mathcal F$ are disjoint.
The regularity assumption implies that the interior of $B'\setminus E$ is nonempty,
hence $\m(B'\setminus E)>0$.
Therefore we have shown that $\m(B'\setminus E)>0$ for every $E\in \mathcal F$ and hence 
by the previous claim with $B'$ in place of $B$, one has
\[
\sum_{E\in \mathcal F} \H^{d-1}(\partial E\cap B)\ge
\sum_{E\in \mathcal F} \H^{d-1}(\partial E\cap B')=+\infty
\]
as claimed.
If $\Omega$ is also connected, then since $\mathcal F$ has at least two elements,
we obtain that $\bigcup_{E\in \mathcal F} \partial E\cap \Omega$ is not empty.
Finally, from~\eqref{eq:68} one has 
\[
  \sum_{E\in \mathcal F} \H^{d-1}(\partial E\cap\Omega) = +\infty,
\]
proving the last claim.
\end{proof}

\begin{remark}
  It is clear from the proof that the statement of Corollary~\ref{co:2} 
  holds under the slightly weaker assumption that
  $\H^{d-1}(\partial^* A\cap \partial^* B)=0$ for all $A,B\in \mathcal F$
  instead of $\H^{d-1}(\bar A\cap \bar B)=0$,
  where $\partial^*$ stands for the reduced boundary in the sense of De Giorgi.
\end{remark}

\begin{remark}
Corollary~\ref{co:2} is valid in any metric measure space
$(X,\d,\m)$
satisfying the Poincar\'e inequality where for every $x\in X$ 
there is an open ball $U\subset X$ containing $x$ 
such that $(U,\d,\m)$ is doubling.
In particular, it is valid in any $C^2$ smooth Riemannian manifold.
Indeed it is enough to rewrite word-to-word the proof of 
Corollary~\ref{co:2} with balls $B\subset U$.
In the case when $X$ a $C^2$ smooth Riemannian manifold
it is enough to note that over every ball $U$ the curvature 
of $X$ is bounded hence $(U,\d,\m)$ is doubling
(in fact it is enough for this purpose that the curvature be bounded from below).
\end{remark}

% The previous corollary can be compared with
% the results of Larman which says 
% that a packing of balls in the unit cube 
% of $\RR^d$ has an \emph{exponent of packing}
% which is strictly larger than $d-1$. 
% We only show that such exponent
% is greater or equal to $d-1$ but for any general 
% packing by means of strictly convex sets.
\begin{remark}\label{remrem}
An alternative proof of Corollary~\ref{co:2} for a family $\mathcal F$ 
of strictly convex sets could proceed as follows:
\begin{itemize}
  \item
For the planar case $d=2$ one uses the coarea inequality
to state that if the total perimeter of the sets 
in the family is finite 
then almost every line (in fact every except a countable number of lines), say horizontal, intersects the
boundary of the sets in a set of finite $\H^{0}$ measure 
(i.e. in a finite set).
On the other hand if the sets of $\mathcal F$ 
are assumed to be strictly convex then
the lines intersecting only a finite number of sets
in the family should be at most countable:
in fact, the intersection of 
a convex set with a line is a line segment, 
and hence if the line intersects only a finite number of sets, 
then each point of its intersection with the boundary of 
some set of $\mathcal F$ is a point of 
intersection between two different sets of $\mathcal F$, 
which are countably many in total. This contradiction shows that 
the total perimeter of the sets in the family is infinite
for planar packing of strictly convex sets (even in the case 
when the ambient set is not convex). 
\item For the general
space dimension $d\geq 2$, again assuming that the total perimeter
of the packing is finite, by coarea inequality we have 
that the intersection of almost every hyperplane, say, horizontal, 
with the boundary of the sets in the family, has finite $\H^{d-2}$ measure. 
If the ambient set is convex, then inside every such hyperplane
we have a packing of a convex set by strictly convex sets. 
Proceeding by backward induction on the dimension we arrive at a contradiction.
Convexity of the ambient set is clearly essential for the argument to work in
the case $d > 2$.
\end{itemize}
\end{remark}

\begin{remark}
%if $\Omega is connected,
%without any convexity assumption on the sets of $\mathcal F$,
Under the assumptions of Corollary~\ref{co:2}, if $\Omega$ is connected and every set in $\mathcal F$ is regular in the sense of this Corollary,
reasoning as in Remark~\ref{remrem} one can prove that
\begin{equation}\label{eq:69}
\sum_{E\in \mathcal F} (\diam E)^{d-1} = +\infty,
\end{equation}
where $\diam E$ denotes the diameter of $E$.
%Indeed, $\H^{d-1}(\Omega\cap \partial E)>0$ for every $E\in \mathcal F$,
%%since $\Omega$ is connected, 
%and hence, without loss of generality, we can assume that 
%$p_1( \Omega \cap \partial E)$ has positive $(d-1)$-dimensional Lebesgue measure
%$\L^{d-1}$ for all $E\in \mathcal F$,  
%where $p_1$ denotes the orthogonal projection on the hyperplane $e_1^\perp$.
Indeed, since $\Omega$ is connected,
we can choose a ball $B\subset \Omega$ such that $V:=\bigcup_{E\in \mathcal F} \partial E\cap B$ is nonempty.
%with $\mathcal H^{d-1}(V)>0$. 
Notice  that, if $x\in \partial E$ for some $E\in\mathcal F$, then every neighborhood of $x$ intersects
at least two sets in $\mathcal F$. 
Therefore, 
up to a suitable rotation,
we can assume that there exists a set $L_1$ of positive measure of lines $\ell$ parallel to the first coordinate direction $e_1$, which intersect in $B$ at least two sets of $\mathcal F$ (hence they also intersect the set $V$).
Here by measure of the set of lines $L_1$ we mean the 
$\L^{d-1}$ measure of its orthogonal projections on the hyperplane $e_1^\perp$ perpendicular to $e_1$.
In particular $\L^{d-1}\left(p_1\left(V\right)\right) >0$,
where $p_1$ denotes the orthogonal projection on $e_1^\perp$.
Since 
\[
  \L^{d-1}(p_1(\bar E))
    \le \omega_{d-1}\diam (p_1(\bar E))^{d-1}
    \le \omega_{d-1}(\diam \bar E)^{d-1}
    = \omega_{d-1}(\diam E)^{d-1},
\]
if \eqref{eq:69} is not satisfied, then, for every $\eps>0$, there is a cofinite subfamily $\mathcal F_\eps\subset\mathcal F$
such that 
\[
\sum_{E\in \mathcal F_\eps}\L^{d-1}(p_1(\bar E)) <\eps.
\]
This implies that there exists a set $L_2\subset L_1$ of positive measure of lines $\ell$ which
do not intersect the closures of the sets in $\mathcal F_\eps$, but  intersect at least two sets in $\mathcal F\setminus\mathcal F_\eps$.
%and such that $\ell\cap \bigcup_{E\in \mathcal F\setminus\mathcal F_\eps} E\neq \emptyset$. 
%Here by measure of a set of lines we mean the $\L^{d-1}$ measure of their orthogonal projections on $e_1^\perp$. \\
%Since every set in $\mathcal F$ may be supposed bounded (otherwise there is nothing to prove), then every line in this set
%intersecting an $E\in \mathcal{F}\setminus \mathcal{F}_\varepsilon$ intersects also $\partial E$. 
Notice that  
$\ell\cap B\subset \bigcup_{E\in \mathcal F\setminus\mathcal F_\eps} \bar E$
for almost every line $\ell$ parallel to $e_1$, since otherwise, by Fubini Theorem, the volume of $B\setminus \bigcup_{E\in \mathcal F} \bar E$ would be strictly positive, contradicting~\eqref{eq:m}.
As a consequence, for almost every line $\ell\in L_2$, since $\ell\cap B$ is connected,
there exist two sets $U, W$ in $\mathcal F\setminus\mathcal F_\eps$, 
possibly depending on $\ell$, such that $\ell \cap B \cap \partial U\cap \partial W$ is nonempty (otherwise $\ell\cap B$ would be a union of a finite disjoint family of relatively closed sets).
Finally, since the family $\mathcal F\setminus\mathcal F_\eps$ is finite,
we can find two sets $E_1, E_2$ in $\mathcal F\setminus\mathcal F_\eps$ and a subset $L_3\subset L_2$ of positive measure such that
$\ell \cap B \cap \partial E_1\cap \partial E_2$ is nonempty for every $\ell\in L_3$, contradicting~\eqref{eq:0967}.
\end{remark}

\bibliographystyle{amsplain}
%\bibliography{../share/math/biblio}
%\input{biblio.bbl}
\providecommand{\bysame}{\leavevmode\hbox to3em{\hrulefill}\thinspace}
\providecommand{\MR}{\relax\ifhmode\unskip\space\fi MR }
\providecommand{\MRhref}[2]{%
  \href{http://www.ams.org/mathscinet-getitem?mr=#1}{#2}
}
\providecommand{\href}[2]{#2}

\end{document}